\pdfoutput=1
\documentclass[reqno,oneside]{amsart}

\usepackage{xcolor}

\usepackage{imakeidx}
\makeindex[name=general]

\usepackage[T1]{fontenc}
\usepackage[utf8]{inputenc}
\usepackage[pagebackref]{hyperref}
\usepackage{mathtools,lmodern,mathabx,bookmark,mleftright,cite}
\usepackage[nameinlink]{cleveref}
\usepackage[ocgcolorlinks]{ocgx2}
\usepackage{graphicx, shuffle, amsaddr}
\usepackage[paper=a4paper,margin={3.5cm, 3cm}]{geometry}
\hypersetup{citecolor=blue}
\usepackage[mathlines]{lineno}
\usepackage[inline]{enumitem}

\newlength{\eX}
\newcommand{\orcid}[1]{\raisebox{-0.4ex}{\resizebox{!}{1.5\eX}{\href{https://orcid.org/#1}{\includegraphics{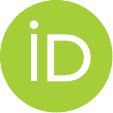}}}}}

\usepackage[linecolor=white,backgroundcolor=white,bordercolor=white,textsize=tiny]{todonotes}
\let\todon\todo
\renewcommand{\todo}[1]{\todon[inline,color=green!40]{\color{blue}{#1}}}

\theoremstyle{plain}
\newtheorem{theorem}{Theorem}[section]
\newtheorem{lemma}[theorem]{Lemma}
\newtheorem{proposition}[theorem]{Proposition}
\newtheorem{corollary}[theorem]{Corollary}

\newtheorem{remark}[theorem]{Remark}

\theoremstyle{definition}
\newtheorem{definition}[theorem]{Definition}

\newtheorem{example}[theorem]{Example}

\newcommand\letter[1]{\ensuremath{\textcolor{cyan}{\mathtt{#1}}}}

\DeclareMathOperator{\area}{\mathtt{area}}
\def\Hqsh{\ensuremath{H_{\mathrm{qsh}}}}
\newcommand*{\R}{{\mathbb R}}
\newcommand\N{\mathbb N}
\def\e{\mathbf e}

\DeclarePairedDelimiter{\dbra}{[\mkern-3.5mu[}{]\mkern-3.5mu]}
\DeclareMathOperator{\DS}{S}
\DeclareMathOperator{\ISS}{ISS}
\def\qsh{\mathrm{qsh}}

\usepackage{etoolbox}
\newcommand*\linenomathpatchAMS[1]{%
    \expandafter\pretocmd\csname #1\endcsname {\linenomathAMS}{}{}%
    \expandafter\pretocmd\csname #1*\endcsname{\linenomathAMS}{}{}%
    \expandafter\apptocmd\csname end#1\endcsname {\endlinenomath}{}{}%
    \expandafter\apptocmd\csname end#1*\endcsname{\endlinenomath}{}{}%
}
\expandafter\ifx\linenomath\linenomathWithnumbers
    \let\linenomathAMS\linenomathWithnumbers
    \patchcmd\linenomathAMS{\advance\postdisplaypenalty\linenopenalty}{}{}{}
\else
    \let\linenomathAMS\linenomathNonumbers
\fi
\linenomathpatchAMS{align}

\mleftright

\begin{document}

\title[Generalized iterated-sums signatures]{Generalized iterated-sums signatures}

\author[J. Diehl]{Joscha Diehl \orcid{0000-0002-8281-7203}}
\address{Universit\"at Greifswald\\ Institut f\"ur Mathematik und Informatik\\ Walther-Rathenau-Str.~47\\ 17489 Greifswald, Germany.}
\email{joscha.diehl@uni-greifswald.de}
\urladdr{https://diehlj.github.io}
\author[K. Ebrahimi-Fard]{Kurusch Ebrahimi-Fard \orcid{0000-0001-5180-5835}}
\address{Norwegian University of Science and Technology -- NTNU\\ Department of Mathematical Sciences\\ 7491 Trondheim, Norway.}
\email{kurusch.ebrahimi-fard@ntnu.no}
\urladdr{https://folk.ntnu.no/kurusche}
\author[N. Tapia]{Nikolas Tapia \orcid{0000-0003-0018-2492}}
\address{Weierstrass Institute\\ Mohrenstr.~39\\ 10117 Berlin, Germany\\ and\\Technische Universität Berlin\\ Str.~des 17.~Juni 136\\ 10623 Berlin, Germany.}
\email{tapia@wias-berlin.de}
\urladdr{https://wias-berlin.de/people/tapia}


\date{\today}

\begin{abstract}
	We explore the algebraic properties of a generalized version of the iterated-sums signature,
    inspired by previous work of F.~Király and H.~Oberhauser.
    In particular, we show how to recover the character property of the associated linear map on
    the tensor algebra by considering a deformed quasi-shuffle product of words.
    We introduce three non-linear transformations on iterated-sums signatures, close in spirit to Machine
    Learning applications, and show some of their properties.
\end{abstract}

\maketitle


\tableofcontents

\section{Introduction}
\label{sec:intro}

Recently, a series of papers \cite{DET2020,DET2020b,KiralyOberhauser2019,LyonsOberhauser2017} have highlighted the importance of \emph{signature-like objects}
(following T.~Lyons' nomenclature) for capturing features of sequentially ordered data.
Part of the reason for their success is that these transformations posses a universality property,
meaning that they are able to approximate arbitrary\footnote{See, for example, \cite[Proposition A.6]{kidger2019deep} for a precise statement.} nonlinear mappings on sequence space
by linear functionals on feature space.
They can also be efficiently computed thanks to an inherent recursive structure.

Both properties can be succinctly described by using Hopf-algebraic language, which has by now
become standard in the field. 
For the so-called iterated-integrals signature, the underlying Hopf algebra is the space of words together with the commutative \emph{shuffle product} \cite{Ree1958,Reutenauer1993} and the noncocommutative deconcatenation coproduct.
On the other hand, replacing integrals by sums, we obtain the iterated-sums signature, which is defined over the commutative \emph{quasi-shuffle algebra} on words \cite{Cartier1972,Gaines1994,Hoffman2000,NewRad1979}. Equipped with the aforementioned deconcatenation coproduct, the latter becomes a Hopf algebra.
Shuffle and quasi-shuffle products algebraically encode integration and summation by parts for iterated integrals and sums respectively. 
In both cases the properties mentioned in the preceding paragraph amount to saying that the corresponding signature-like maps are
characters, i.e., algebra morphisms, and that they satisfy Chen's relation.

Recently, F.~Király and H.~Oberhauser \cite{KiralyOberhauser2019} introduced a higher order version
of the iterated-sums signature as a way of approximating the iterated-integrals signature.
The main disadvantage of this generalization is that the character property
is lost and consequently the universality property ceases to hold.
In the paper at hand we unfold the algebraic underpinning of the definition of Király and Oberhauser's  higher order iterated-sums signature, which permits further generalization to arbitrary nonlinearities, as opposed to the more restricted class of exponential-type nonlinearities underlying previous approaches.
Moreover, we show that the generalized  iterated-sums signature enjoys a character property with respect to a different Hopf algebra defined on words in terms of a modified quasi-shuffle product and the deconcatenation coproduct; in fact, we show that the algebraic structure actually depends on the selected nonlinear transformation.
We come back to this topic in \Cref{sss:KO}.

Thanks to the more general approach, we are able to introduce three new transformations of the
iterated-sums signature. The first transformation is obtained by applying a tensorized nonlinear transformation
to each time slice, the second one is constructed by applying a polynomial map to increments,
whereas the third is obtained by first transforming the data and then considering its increments. We
show that these transformations can be expressed in terms of the un-transformed iterated-sums
signature. In the third case, this rewriting procedure generalizes earlier work by L.~Colmenarejo and R.~Preiß on iterated-integrals signatures defined with respect to paths transformed by polynomial maps \cite{CP2020}.

\smallskip

The rest of the paper is organized as follows. In \Cref{sec:quasi-shuffle}, we review the algebraic
foundation of our construction, i.e., the notion of quasi-shuffle Hopf algebra.
In \Cref{sec:pouterVarying}, we introduce the generalized iterated-sums signature map and provide a
complete description of its most important properties using the developments of the previous
section.

\medskip

{\textbf{Acknowledgments}}: We thank the referee for pertinent remarks and observations improving our understanding of certain algebraic aspects, which ameliorated the presentation.
We also thank Rosa Preiss for helpful comments. 
The second author is supported by the Research Council of Norway through project 302831 ``Computational Dynamics and Stochastics on Manifolds'' (CODYSMA).
The third author is funded by the Deutsche Forschungsgemeinschaft (DFG, German Research Foundation) under Germany's Excellence Strategy – The Berlin Mathematics Research Center MATH+ (EXC-2046/1, project ID: 390685689)


\section{Quasi-shuffle Hopf algebra}
\label{sec:quasi-shuffle}

In this section we briefly recall the notion of commutative quasi-shuffle product, the algebraic framework present in \cite{DET2020}. However, we shall emphasise the refined viewpoint based on the notion of half-shuffle product. Readers interested in the details are directed to further references
\cite{EFP2020,Hoffman2000,FP2020,FPT2016,HofIha2017}.

Following Foissy and Patras \cite{FP2020} we define the notion of commutative quasi-shuffle algebra over the base field $\mathbb{R}$.

\begin{definition}
    A commutative quasi-shuffle algebra $(\mathcal A,{\succ},{\bullet})$ consists of a nonunital commutative $\mathbb{R}$-algebra
    $(\mathcal A,{\bullet})$ equipped with a linear \emph{half-shuffle product} ${\succ}\colon \mathcal A \otimes \mathcal A \to \mathcal A$ satisfying
\begin{align}
    x \succ (y \succ z) &= (x * y) \succ z, \label{Qsh1}\\
    (x \succ y) \bullet z &= x \succ (y \bullet z), \label{Qsh2}\\
    (x \bullet y)\bullet z &= x\bullet (y\bullet z),\label{Qsh3}
\end{align}
where the {\it{quasi-shuffle product}} ${\ast}\colon \mathcal A \otimes \mathcal A \to \mathcal A$ is defined as
\begin{equation}
\label{qshproduct}
x * y \coloneq x \succ y + y \succ x + x \bullet y.
\end{equation}
\label{def:comQSh}
\end{definition}

One verifies that relations \eqref{Qsh1}, \eqref{Qsh2} and \eqref{Qsh3} imply that the quasi-shuffle product \eqref{qshproduct} is both commutative and associative. Observe that one may characterise a commutative quasi-shuffle algebra as a space with two commutative products related through a --symmetrized-- half-shuffle product. In the following, quasi-shuffle algebra means commutative quasi-shuffle algebra.

\begin{definition}
A quasi-shuffle morphism between two quasi-shuffle algebras, \((\mathcal A,{\succ},{\bullet})\) and \((\tilde{\mathcal A},{\tilde{\succ}},{\tilde{\bullet}})\), is a linear map \(\Lambda\colon \mathcal A\to\tilde{\mathcal A}\) satisfying
	\[
		\Lambda(x\succ y)=\Lambda(x)\mathbin{\tilde{\succ}}\Lambda(y),\quad\Lambda(x\bullet
		y)=\Lambda(x)\mathbin{\tilde{\bullet}}\Lambda(y),
	\]
	for all \(x,y\in \mathcal A\).
	\label{def:qshmorph}
\end{definition}

Any quasi-shuffle morphism is an algebra morphism between quasi-shuffle algebras, that is, $\Lambda (x *
y) = \Lambda(x)\mathbin{\tilde{*}} \Lambda(y)$.
A quasi-shuffle algebra $(\mathcal A,\succ,\bullet)$ has a unital extension.
Indeed, set $\bar{\mathcal A}\coloneq\R1\oplus \mathcal A$ and define for $a \in \mathcal A$: $1
\bullet a = a \bullet 1\coloneq 0$, $1 \succ a \coloneq a$ and $a \succ 1\coloneq 0$.
Note that the product $1 \succ 1$ as well as $1 \bullet 1$ are excluded and so the identity $1 \ast 1 \coloneq 1$ must be imposed in addition. This turns $\bar{\mathcal A}$ into a unital algebra with quasi-shuffle product
\(\ast\).
Furthermore, any quasi-shuffle morphism can be extended as a unital algebra morphism between the corresponding unitizations; in particular, this
extension will preserve the units.

If the commutative algebra $(\mathcal A,\bullet)$ in \Cref{def:comQSh} has a trivial product, i.e., \(x\bullet y=0\) for all \(x,y\in \mathcal A\), then the notion of commutative quasi-shuffle algebra reduces to that of commutative shuffle algebra, which is defined solely in terms of relation \eqref{Qsh1}. In this case the commutative and associative product \eqref{qshproduct} is called shuffle product.

\begin{remark}
We note that different terminologies are present in the literature. Commutative shuffle and quasi-shuffle algebras are also known as Zinbiel and commutative tridendriform algebras, respectively. The noncommutative generalizations of both shuffle and quasi-shuffle algebra, are also known as dendriform and tridendriform algebras, respectively.
We follow  {\rm{\cite{FP2020}}}, where the preference for the terminology used in this work is explained. 
\end{remark}

Our main example provides also the paradigm of commutative quasi-shuffle algebra, i.e., the free commutative quasi-shuffle algebra.  
Let \(A=\{\letter 1,\dotsc,\letter d\}\) be a finite alphabet and consider the reduced symmetric algebra
\(S(A)\) over the vector space spanned by it.
By way of explanation, \(S(A)\) is the space spanned by words in commuting letters from \(A\); here
\emph{reduced} means that we do not suppose \(S(A)\) to have a unit, i.e.,~we consider only the
augmentation ideal of the standard symmetric algebra over \(A\). Keeping up with our previous convention \cite{DET2020}, we denote the commutative product in \(S(A)\) by square brackets.
Neither endow we \(S(A)\) with any additional algebraic structure other than its product.
Finally we recall that \(S(A)\) has a natural grading
\[ 
	S(A)=\bigoplus_{n=1}^\infty S^nA, 
\]
where \( S^nA \) is spanned by products of the form \([\letter i_1\dotsm\letter i_n]\) with
\(\letter i_1,\dotsc,\letter i_n\in A\).
We denote this basis by \(\mathfrak A_n\).
It is well known that \(\dim S^nA=\binom{d+n-1}{n}\).
Furthermore, the set
\[
        \mathfrak A=\bigcup_{n=1}^\infty\mathfrak A_n
\]
constitutes a basis for \(S(A)\).

Now, we let \(H\coloneq T(S(A))\) be the unital tensor algebra over \(S(A)\).
As a vector space, it is the direct sum
\[
	H=\bigoplus_{n=0}^\infty S(A)^{\otimes n}=\bigoplus_{n=0}^\infty H_n, 
\]
where \(H_0=\R\e\) and
\[
	H_n = \bigoplus_{k=1}^n\left( \bigoplus_{i_1+\dotsb+i_k=n}S^{i_1}A\otimes\dotsm\otimes
        S^{i_k}A \right).
\]
We also set
\[
    H^+=\bigoplus_{n=1}^\infty H_n.
\]

In the following we will use the word notation for elements in $H$.
For example, when \(d=2\), a generic element of \(H\) might look like
\[
    \sqrt{3}\,[\letter{1}][\letter{2}]+\frac{\pi^2}{6}[\letter{12}]+2\,[\letter1][\letter2][\letter1].
\]

Concatenation, written by juxtaposition of symbols, is the standard product on $H$. In particular, \(H\) inherits a grading from \(S(A)\) in this way, which we call the \emph{weight} and denote by \(|\cdot|\).
The length of a word $w=s_1 \cdots s_k \in H$ is defined to be \(\ell(w)=k\).
In \cite{DET2020} we show that
\[
    \sum_{n=0}^\infty t^n\dim H_n=\frac{(1-t)^d}{2(1-t)^d-1}=1+dt+\frac{d(3d+1)}{2}t^2+\frac{d(13d^2+9d+2)}{6}t^3+\dotsb.
\]
It can also be shown that the dimensions satisfy the following recursion:
\[
    \dim H_n=\sum_{j=1}^{n}\binom{d+j-1}{d-1}\dim H_{n-j},\quad\dim H_0=1.
\]

The standard basis for \(H\) is the set of words over \(\mathfrak A\), here denoted by \(\mathfrak
A^*\).
We endow \(H\) with an inductively defined product obtained from the bracket product of \(S(A)\): \(\e\star u\coloneq u\eqcolon
u\star\e\) and
\begin{equation}
\label{qshufprod}
ua\star vb
=(u\star vb)a
+(ua\star v)b
+(u\star v)[ab]
\end{equation}
for \(u,v\in H\) and \(a,b\in S(A)\).
Hoffman \cite{Hoffman2000} called \eqref{qshufprod} quasi-shuffle product and showed that it is commutative and associative as well as compatible with the deconcatenation coproduct \(\Delta_{\scriptscriptstyle{\text{dec}}}\colon H\to H\otimes H\), defined on basis elements \(u=u_1\cdots u_n\in\mathfrak A^*\) by
\[
	\Delta_{\scriptscriptstyle{\text{dec}}} u=u\otimes\mathbf e+\mathbf e\otimes u+\sum_{j=1}^{n-1}u_1\dotsm u_{j}\otimes u_{j+1}\dotsm u_n,
\]
and the counit \(\varepsilon\) determined by the grading, so that
\(H_{\qsh}\coloneq(H,\star,\Delta_{\scriptscriptstyle{\text{dec}}},\e,\varepsilon,\alpha)\) is a Hopf algebra. See also \cite{NewRad1979} and \cite{Gaines1994}.

The non-unital quasi-shuffle algebra $\Hqsh^+$ carries a commutative quasi-shuffle structure \cite{L2007}, defined recursively
by
\begin{equation}
    u\succ va\coloneqq (u\star v)a,
    \quad
    ua\bullet vb=(u\star v)[ab].
    \label{eq:qshrec}
\end{equation}
It is such that the unitization of \(\Hqsh^+\) is isomorphic to \(\Hqsh\) as a unital commutative algebra. Observe that any word $w=s_1 \cdots s_k \in H$ can be written using the half-shuffle product defined in \eqref{eq:qshrec}
\begin{equation*}
    w = (\cdots ((s_1 \succ s_2)\succ s_3) \cdots )\succ s_k.
\end{equation*}

\begin{remark}
It is natural to consider the relation between the deconcatenation coproduct $\Delta_{\scriptscriptstyle{\text{dec}}}$ and half-shuffle as well as the
$\bullet$ products. It turns out that they form what is known as quasi-shuffle bialgebra. See reference {\rm{\cite{FP2020}}} for more details.
\end{remark}

Finally, we recall the following important result due to Loday \cite[Theorem 2.5]{L2007}.

\begin{theorem}
    The free commutative unital quasi-shuffle algebra over \(\R^d\) is isomorphic to \(\Hqsh\).
    \label{thm:loday}
\end{theorem}

The (algebraic) dual space \(\Hqsh'\) can be identified with formal word series 
\[
        \mathrm T=\sum_{w\in\mathfrak A^*}\langle\mathrm T,w\rangle w,
\]
in the sense that there is an isomorphism between such formal series and elements of \(H'\).
The convolution product of two maps \(\mathrm R,\mathrm T\in H'\) is defined by
\[
        \mathrm R\mathrm T\coloneq\sum_{w\in\mathfrak A^*} \langle\mathrm R\otimes\mathrm T,\Delta_{\scriptscriptstyle{\text{dec}}} w\rangle w=\sum_{w\in\mathfrak A^*}\left(
        \sum_{uv=w}\langle\mathrm R,u\rangle\langle\mathrm T,v\rangle \right)w.
\]
Observe that this product is associative but not commutative.


\subsection{Deformed quasi-shuffle products}
\label{sse:twist}

In this section, we describe how coalgebra automorphisms of the coalgebra \((H,\Delta_{\scriptscriptstyle{\text{dec}}})\) can be used to define new quasi-shuffle structures on \(H\) by
deforming Hoffman's quasi-shuffle product.
We recall that a coalgebra morphism between two coalgebras \((A,\Delta_A,\varepsilon_A)\) and \((B,\Delta_B,\varepsilon_B)\) is a linear map
\(\Psi\colon A\to B\) such that \(\Delta_B\circ\Psi=(\Psi\otimes\Psi)\circ\Delta_A\) and \(\varepsilon_B\circ\Psi=\varepsilon_A\).

\begin{remark}
    \label{rem:coalgebraEndomorphisms}
	All possible coalgebra endomorphisms of \((H,\Delta_{\scriptscriptstyle{\text{dec}}})\) have been characterized by Foissy, Patras and Thibon
   {\rm{\cite[Corollary 2.1]{FPT2016}}}.
	In our context, the result is that this class is isomorphic to the class of linear maps from \(H^+\) to \(S(A)\).
	The isomorphism is described in the following way: any linear map \(\zeta\colon H^+\to S(A)\) induces a unique
    coalgebra endomorphism \(\varphi_\zeta\) of
	\((H,\Delta_{\scriptscriptstyle{\text{dec}}})\) via
	\[
		\varphi_\zeta(a_1\dotsm a_n)=\sum_{k=1}^n\sum_{i_1+\dotsb+i_k=n}\zeta(a_1\dotsm a_{i_1})\dotsm \zeta(a_{i_1+\dotsb+i_{k-1}+1}\dotsm a_n).
	\]

	Conversely, any coalgebra endomorphism \(\varphi\colon(H,\Delta_{\scriptscriptstyle{\text{dec}}}) \to (H,\Delta_{\scriptscriptstyle{\text{dec}}})\) is uniquely determined by the
    knowledge of \(\zeta=\pi\circ\varphi\), where
	\(\pi\) denotes the projection of \(H\) onto \(S(A)\), which is orthogonal to \(\R\) and \(S(A)^{\otimes k}\) for \(k\ge 2\).
	In general, the result remains valid if \(S(A)\) is replaced by any commutative algebra.
	Moreover, in {\rm{\cite{FPT2016}}} it is also shown that \(\varphi\) is an automorphism if and only if the restriction of \(\zeta=\pi\circ\varphi\) to \(S(A)\) is a linear isomorphism.
\end{remark}

From now on, \(\Psi\) will denote a coalgebra automorphism of \((H, \Delta_{\scriptscriptstyle{\text{dec}}})\).
The following three assertions can be derived straightforwardly by transporting all algebraic structure through the mapping $\Psi$.
\begin{proposition}
\label{prop:deformedqSh}
\index[general]{succPsi@$\succ_\Psi$}
\index[general]{bulletPsi@$\bullet_\Psi$}
The space $H^+$ equipped with the deformed products
$$
        u \succ_\Psi v \coloneq\Psi^{-1}(\Psi(u)\succ \Psi(v))
        \qquad
        u \bullet_\Psi v \coloneq\Psi^{-1}(\Psi(u)\bullet \Psi(v))
$$
is a commutative quasi-shuffle algebra. The deformed quasi-shuffle product
\index[general]{starpsi@$\star_\Psi$, deformed product}
\begin{equation}
\label{twqshuf}
    u\star_\Psi v\coloneq\Psi^{-1}(\Psi(u)\star\Psi(v)).
\end{equation}
is associative and commutative.
\end{proposition}

\begin{proposition}
The space \(H_\Psi\coloneq(H,\star_\Psi,\Delta_{\scriptscriptstyle{\text{dec}}},\Psi^{-1}(\e),\varepsilon)\) is a connected graded Hopf algebra.
\end{proposition}

\begin{proposition}
    \label{prop:psif}
    The map \(\Psi\colon H_\Psi\to H_{\qsh}\) is a Hopf algebra isomorphism.
\end{proposition}

\bigskip

\begin{remark}~
    \begin{enumerate}[leftmargin=1.8em]
        \item It is well known that when \(\Psi\) is Hoffman's exponential, the deformed quasi-shuffle product \(\star_\Psi\)
            coincides with the classical shuffle product on \(H\).
            We note, however, that the splitting given in \Cref{twqshuf} \emph{does not} coincide with its
            standard half-shuffle decomposition; in particular, the associative product \(\bullet_\Psi\) is
            non-zero.
            This is consistent with the fact, noted in {\rm{\cite[Remark 5.10]{DET2020}}}, that Hoffman's exponential is
            \emph{not} a half-shuffle morphism.
        \item Defining the area\footnote{This terminology comes from the interpretation of half-shuffles as integration
            operators (see e.g. \cite[Section 5.1]{DET2020})} operations $\area_\Psi(u,v)\coloneq u \succ_\Psi v - v \succ_\Psi u$ and
            $\area(u,v)\coloneq u \succ v - v \succ u$, for $u,v \in H$, we see immediately that
        \[
            \Psi\bigl( \area_\Psi(u,v)\bigr)=\area\bigl(\Psi(u),\Psi(v)\bigl).
        \]
        This is an isomorphism property of Tortkara algebras, introduced by Dzhumadil'daev in 2007 \cite{dzhu}.
\end{enumerate}
\end{remark}


\subsection{Coalgebra morphisms induced by formal power series}
\label{sss:coalg.pseries}

We now introduce a special class of coalgebra automorphisms of \(H_{\qsh}\), described in \cite{Hoffman2000,HofIha2017}.
Recall that composition of an integer \(n\ge 1\) refers to a sequence \( I=(i_1,\dotsc,i_k) \) of positive integers
such that \(i_1+\dotsb+i_k=n\). We write \(\mathcal C(n)\) for the set of all compositions of \(n\). For any word \(w=s_1\dotsm s_n\in H\) and composition \(I=(i_1,\dotsc,i_k)\in\mathcal C(n)\) we define the word $I[w] \in H$
\[
        I[w]\coloneq [s_1\dotsm s_{i_1}][s_{i_1+1}\dotsm s_{i_1+i_2}]\dotsm[s_{i_1+\dotsb+i_{k-1}+1}\dotsm s_n].
\]

Formal diffeomorphisms \(f\in t\mathbb R\dbra{t}\) induce linear endomorphisms of \(H\) in the following way: suppose
that
\[
        f(t) = \sum_{n=1}^\infty c_nt^n,
\]
and define
\index[general]{Psif@$\Psi_f$}
\begin{align}
    \label{eq:Psif}
    \Psi_f(w)\coloneq\sum_{I\in\mathcal C(\ell(w))}c_{i_1}\dotsm c_{i_k}I[w].
\end{align}

\begin{remark}
    \label{rem:specialCase}
    Note that this a special case of \Cref{rem:coalgebraEndomorphisms} with
    \begin{align*}
        \zeta( a_1 \cdots a_k ) := c_k [a_1 \cdots a_k].
    \end{align*}
\end{remark}

Formal diffeomorphisms with \(c_1\neq 0\) are invertible with respect to the composition of formal power
series.
In this case, \(\Psi_f\) is an automorphism, and it can be shown that
\(\Psi_f \circ \Psi_g =\Psi_{f \circ g}\), and in particular \(\Psi_f^{-1}=\Psi_{f^{-1}}\), \cite{HofIha2017}.
Finally, we observe that \(\Psi_f\) is always a coalgebra morphism \cite{FPT2016}; this means that the identity
\[
        \Delta_{\scriptscriptstyle{\text{dec}}}\circ\Psi_f=(\Psi_f\otimes\Psi_f)\circ\Delta_{\scriptscriptstyle{\text{dec}}}
\]
holds. Moreover, by definition, the map \(\Psi_f\) is graded, that is, \(\Psi_f(H_n)\subset H_n\) for all \(n\ge 0\).
Therefore, any invertible diffeomorphism induces a deformed quasi-shuffle algebra which we denote by \(H_f\).

In \Cref{ssec:generalizedSumSig} we restrict ourselves to deformed quasi-shuffles induced from invertible formal diffeomorphisms, since in this case the map 
\(\Psi_f\) has a direct interpretation in terms of the iterated-sums signature.



\section{Iterated-sums signatures}
\label{sec:pouterVarying}

We start this section by recalling the definition of the iterated-sums signature introduced in
\cite{DET2020}.
Fix integers \(d\ge 1\) and \(N>0\). A \emph{\(d\)-dimensional time series of length \(N\)} is a sequence of vectors
\(x=(x_0,\dotsc,x_{N-1})\in(\mathbb R^d)^N\).
The following notation for elements in the time series \(x\) is put in place:
\[
    x_j=(x_j^{(1)},\dotsc,x_j^{(d)}),
\]
and it is extended to include brackets in \(\mathfrak A\) by defining
\begin{equation}
    x_j^{[\letter i_1\dotsm\letter i_n]}\coloneq x_j^{(i_1)}\dotsm x_j^{(i_n)}.
    \label{eq:bracketext}
\end{equation}
Given a \(d\)-dimensional time series of length \(N\), its increment series, denoted by \(\delta x\), is also a \(d\)-dimensional time series of
length \(N-1\) with entries defined by \(\delta x_k\coloneq x_{k}-x_{k-1}\).
We also denote, for \(a=[\letter i_1\dotsm\letter i_n]\in\mathfrak A\),
\[
    \delta x^a_j\coloneq\left( x_{j}^{(i_1)}-x_{j-1}^{(i_1)} \right)\dotsm\left( x_{j}^{(i_n)}-x_{j-1}^{(i_n)} \right).
\]

\begin{definition}
Let \(x\) be a \(d\)-dimensional time series, and denote by \(\delta x\) its increment series. The
\emph{iterated-sums signature} of \(x\) is the two-parameter family $(\ISS(x)_{n,m}:0\le n\le m\le
N)$ of linear maps in \(\Hqsh'\) such that
\(\ISS(x)_{n,n}=\varepsilon\), and defined recursively by
\(\langle\ISS(x)_{n,m},\e\rangle\coloneq 1\), and
\[
    \langle \ISS(x)_{n,m},a_1\dotsm a_p\rangle
    \coloneq\sum_{j=n+1}^{m}\langle\ISS(x)_{n,j-1},a_1\dotsm a_{p-1}\rangle\ \delta x_{j}^{a_p}
\]
for all words \(a_1\dotsm a_p\in\mathfrak A^*\).
\label{def:ISS}
\end{definition}

We recall that, as a formal word series, the map \(\ISS(x)_{n,m}\) can be expressed as the time-ordered product
(defined with respect to the concatenation product in $T(S(A))$)
\begin{equation}
        \ISS(x)_{n,m}=\vec{\prod_{n< j\le m}}\left( \varepsilon+\sum_{a\in\mathfrak A}\delta x_j^a a \right).
        \label{eq:ISS}
\end{equation}

In fact, \cref{eq:ISS} can be seen to arise as the solution to a fixed-point equation in \(\Hqsh\).
Indeed, from \Cref{def:ISS} we see that for any word \(a_1\dotsm a_p\in\mathfrak A^*\),
\begin{align*}
    \delta\langle\ISS(x)_{n,\cdot},a_1\dotsm a_p\rangle_m
    &= \langle\ISS(x)_{n,m},a_1\dotsm a_p\rangle-\langle\ISS(x)_{n,m-1},a_1\dotsm a_p\rangle \\
    &= \langle\ISS(x)_{n,m-1}, a_1\dotsm a_{p-1}\rangle\delta x_m^{a_p}.
\end{align*}
Therefore, the equality between word series
\[
    \delta\ISS(x)_{n,m} = \ISS(x)_{n,m-1}\Phi(\delta x_m),\quad
    \ISS(x)_{n,n}=\varepsilon
\]
holds, where the ``polynomial extension'' map \(\Phi\colon\R^d\to S(\mkern-3.5mu(A)\mkern-3.5mu)\)
(c.f. \cite[eq.~60]{NovThi2004}) is defined by 
\[
    \Phi(z)=\sum_{n=1}^\infty\left( \sum_{\letter i\in A}z^{(i)}[\letter i] \right)^n=\sum_{a\in\mathfrak A}z^aa,
\]
(with respect to the commutative product in $S(\mkern-3.5mu(A)\mkern-3.5mu)$). Note that \(\Phi\) amounts to a geometric series in the completed symmetric algebra \(S(\mkern-3.5mu(A)\mkern-3.5mu)\) -- recall from \Cref{sec:quasi-shuffle} that the bracket \([\cdot\cdot]\) denotes the symmetric
tensor product in \(S(A)\).
This extension has also been considered in a Machine
Learning context by Toth, Bonnier and Oberhauser \cite{TBO2020}.

We now record the two most relevant properties of the iterated-sums signature of a \(d\)-dimensional
time series, shown in \cite[Theorem 3.4]{DET2020}.
\begin{theorem}~
        \begin{enumerate}
                \item For each \(0\le n\le m\le N\), the map \(\ISS(x)_{n,m}\) is a quasi-shuffle algebra character.
                \item (Chen's identity) For any \(0\le n\le n'\le n''\le N\) we have
                        \[
                                \ISS(x)_{n,n'}\ISS(x)_{n',n''}=\ISS(x)_{n,n''}.
                        \]
        \end{enumerate}
        \label{thm:issprop}
\end{theorem}
\begin{remark}
    Property (1) in the above Theorem is the iterated-sums analogue of Ree's shuffle identity for iterated integrals.
\end{remark}

In fact, in light of the commutative quasi-shuffle structure of $\Hqsh$, one can be more precise about the nature of item (1) in \Cref{thm:issprop}.
Let \(\mathfrak X_N\) denote the space of real-valued time series with fixed time horizon \(N\in\mathbb N\).
It carries itself a commutative quasi-shuffle structure, given
by
\[
    (x\succeq y)_k\coloneq\sum_{j=1}^{k}(x_{j-1}-x_0)\delta y_j,
    \qquad 
    (x\odot y)_k \coloneq\sum_{j=1}^{k}\delta x_j\delta y_j.
\]
The corresponding commutative and associative quasi-shuffle product is \((x, y)\mapsto(x_\cdot-x_0)(y_\cdot-y_0)\) in \(\mathfrak X_N\). For a given \(d\)-dimensional time series $x$, we define a map \(\sigma(x)\colon A\to\mathfrak X_N\)
by
\[
    \langle\sigma(x), [\letter i]\rangle_k\coloneq x^{(i)}_k-x^{(i)}_0,\quad 0\le k\le N.
\]
By \Cref{thm:loday}, it admits a unique extension to $\Hqsh$ as a commutative quasi-shuffle
morphism (in the sense of \Cref{def:qshmorph}).

\begin{proposition}
    The unique extension \(\sigma(x)\colon\Hqsh\to\mathfrak X_N\) is such that for all \(0\le k\le N\)
    and words \(w\in\mathfrak A^*\) we have
    \[
        \langle \sigma(x),w\rangle_k=\langle\ISS(x)_{0,k},w\rangle.
    \]
\end{proposition}

\begin{proof}
    We first observe that since \([\letter i_1]\bullet\dotsm\bullet[\letter i_n]=[\letter i_1\dotsm\letter i_n]\) for all
    \(\letter i_1,\dotsc,\letter i_n\in A\), we have
    \begin{align*}
        \langle\sigma(x),[\letter i_1\dotsm\letter i_n]\rangle_k&= (\langle \sigma(x),[\letter
        i_1]\rangle\odot\dotsm\odot\langle\sigma(x),[\letter i_n]\rangle)_k\\
        &= \sum_{j=1}^{k}\delta
        x_j^{(i_1)}\dotsm\delta x_j^{(i_n)}\\
        &= \langle\ISS(x)_{0,k},[\letter i_1\dotsm\letter i_n]\rangle
    \end{align*}
    by \cref{eq:bracketext}.
    This shows the identity for all words of length 1.
    Now, suppose the equality is proven for all words up to length \(p\).
    Any word \(w\) of length \(p+1\) can be decomposed as \(w=ua\) for some \(u\in\mathfrak A^*\) with \(\ell(u)=p\) and
    \(a\in\mathfrak A\).
    Since, from \cref{eq:qshrec} (set \(v=\e\)), \(u\succ a=ua\) for any \(u\in\mathfrak A^*\) and \(a\in\mathfrak A\), we see that
    \begin{align*}
        \langle\sigma(x),ua\rangle_k&= \langle\sigma(x),u\succ a\rangle_k\\
        &= (\langle\sigma(x),u\rangle\succeq\langle\sigma(x),a\rangle)_k\\
        &= \sum_{j=1}^{k}\langle\ISS(x)_{0,j-1},u\rangle\delta x^{a}_{j}\\
        &= \langle\ISS(x)_{0,k},ua\rangle
    \end{align*}
    by \Cref{def:ISS}.
\end{proof}

To construct the maps \(x\mapsto \ISS(x)_{n,k}\) for \(0\le n\le k\le N\), one can follow a similar route, by first
considering the map \(x\mapsto \tilde x=(\tilde x_k:n\le k\le N)\) with \(\tilde
x_k=x_k-x_n\).
The image of \(\mathfrak X_N\) under this map will be denoted by \(\tilde{\mathfrak
X}_{n,N}\).
The quasi-shuffle structure defined above can be transported to \(\tilde{\mathfrak X}_{n,N}\) via this map, i.e., \(\tilde
x\succeq\tilde y\coloneq\widetilde{x\succeq y}\), \(\tilde x\odot\tilde y\coloneqq\widetilde{x\odot
y}\).
It is not difficult to see that then the same procedure applied now to \(\tilde{\mathfrak
X}_{n,N}\) gives rise to \(\ISS(x)_{n,k}\) for \(0\le n\le k<N\).
In particular we obtain the following

\begin{proposition}
    Let \(x\) be \(d\)-dimensional time series and fix \(0\le n\le m\le N\). The identities
    \begin{align}
        \langle\ISS(x)_{n,m},u\succ v\rangle   &= \sum_{n<k\le m} \langle\ISS(x)_{n,k-1},u\rangle \delta\langle\ISS(x)_{n,\cdot},v\rangle_k\label{eq:iss.ids1}\\
        \langle\ISS(x)_{n,m},u\bullet v\rangle &= \sum_{n< k\le m} \delta\langle\ISS(x)_{n,\cdot},u\rangle_{k} \delta\langle\ISS(x)_{n,\cdot},v\rangle_k\label{eq:iss.ids2}
    \end{align}
    hold for all \(u,v\in\Hqsh\).
    \label{prop:iss.ids}
\end{proposition}

\begin{remark}\label{rmk:algo}
    In more concrete terms, the content of \Cref{prop:iss.ids} can be interpreted as computation formulas for the entries of \(\ISS\).
    Indeed, since every symbol \(a\in\mathfrak A\) which is not a single letter can be written as \(a=a'\bullet\letter i\) for some \(a'\in\mathfrak A\) with \(|a'|=|a|-1\) and \(\letter i\in A\), entries indexed by \(\mathfrak A\) can be computed inductively using \cref{eq:iss.ids2}, starting from single letters.
    Next, one notes that every non-empty word \(w\in\mathfrak A^*\) can be decomposed as \(w=w'a=w'\succ a\) for some \(w'\in\mathfrak A^*\) and \(a\in\mathfrak A\) with \(|w|=|w'|+|a|\), entries indexed by words with length greater than 1 can be computed inductively using \cref{eq:iss.ids1} and the values computed in the previous step.
\end{remark}


\subsection{Generalized iterated-sums signatures}
\label{ssec:generalizedSumSig}

Let \(f\in t\mathbb R\dbra{t}\) be a formal diffeomorphism.
In addition to the map \(\Psi_f\) described in \Cref{sss:coalg.pseries},
it induces a transformation on formal word series by
\[
    f_\otimes(\mathrm R)=\sum_{n=1}^\infty c_n\mathrm R^n.
\]

\begin{definition}
    \label{def:generalized}
        Let \(x\) be a \(d\)-dimensional time series and \(f\in t\mathbb R\dbra{t}\) a formal diffeomorphism. The
        \emph{generalized iterated-sums signature} is the family of linear maps \((\ISS^f(x)_{n,m}:0\le
        n\le m\le N)\) defined by
        \index[general]{ISSf@$\ISS^f$, generalized iterated-sums signature}
        \[
                \ISS^f(x)_{n,m}\coloneq\vec{\prod_{n< j\le m}}
                \left( \varepsilon + f_\otimes\left( \sum_{a\in \mathfrak A}\delta x_j^a a \right) \right).
        \]
\end{definition}

We immediately have

\begin{proposition}
        The generalized iterated-sums signature satisfies Chen's property, that is, for any
        \(0\le n\le n'\le n''\le N\)
        \[
                \ISS^f(x)_{n,n'}\ISS^f(x)_{n',n''}=\ISS^f(x)_{n,n''}.
        \]
\end{proposition}

We observe that due to the nonlinear nature of the transformation \(f\) applied inside the product,
expansion of this expression as a proper word series is, in principle, not straightforward.
However we have the following result.

\begin{proposition}
    \label{prp:iss_psi}
        For every \(w\in H\),
        \[
                \langle \ISS^f(x)_{n,m},w\rangle=\langle\ISS(x)_{n,m},\Psi_f(w)\rangle,
        \]
        where $\Psi_f$ is defined in \cref{eq:Psif}.
\end{proposition}

\begin{proof}
        First we observe that, by Chen's property and the fact that \(\Psi_f\) is a coalgebra morphism,
        it suffices to show that the equality holds when \(m=n+1\), i.e., we only need to show that
        \[
            \left\langle f_\otimes\left( \sum_{a\in\mathfrak A}\delta x_n^aa
            \right),w\right\rangle=\sum_{a\in\mathfrak A}\delta x_n^a\langle
                a,\Psi_f(w)\rangle.
        \]
        Moreover, since the identity is linear in \(w\), we can further restrict ourselves to the case
        \(w\in\mathfrak A^*\).

        Now, by definition,\begin{align*}
            f_\otimes\left( \sum_{a\in\mathfrak A}\delta x_n^a a \right)&= \sum_{m=1}^\infty c_m\left(
                \sum_{a\in\mathfrak A}\delta x_n^a a
                \right)^m\\
                &= \sum_{m=1}^\infty c_m\sum_{a_1,\dotsc,a_m\in\mathfrak A}\delta x_n^{a_1}\dotsm\delta x_n^{a_m}a_1\dotsm
                a_m.
        \end{align*}
        Since \(\delta x_n^{a_1}\dotsm\delta x_n^{a_m}=\delta x_n^{[a_1\dotsm a_m]}\) we obtain that, if
        \(w=a_1\dotsm a_m\in\mathfrak A^*\),
        \[
            \left\langle f_\otimes\left( \sum_{a\in\mathfrak A}\delta x_n^a a \right),w\right\rangle=c_m\delta x_n^{[a_1\cdots
                a_m]}.
        \]

        On the other hand, we have
        \begin{align*}
            \left\langle\sum_{a\in\mathfrak A}\delta x_n^a a,\Psi_f(w)\right\rangle
            &= \sum_{a\in\mathfrak A}\delta x_n^a\langle a,\Psi_f(w)\rangle\\
                &= \sum_{a\in\mathfrak A}\delta x_n^a\sum_{J\in\mathcal C(m)}c_{i_1}\dotsm c_{i_k}\langle a,I[w]\rangle.
        \end{align*}
        However, in the last sum the only word of length 1 of the form \(I[w]\) is
        \((m)[w]=[a_1\dotsm a_m]\). Therefore
        \[
            \left\langle\sum_{a\in\mathfrak A}\delta x_n^a a,\Psi_f(w)\right\rangle= c_m\delta x_n^{[a_1\cdots a_m]}
        \]
        and the equality is proven in this case.
\end{proof}

An application of \Cref{prop:psif} yields
\begin{corollary}
        The generalized iterated-sums signature is a character over the  Hopf
        algebra \(H_f\) with deformed quasi-shuffle product.
        \label{crl:twchar}
\end{corollary}

Also, since \(\Psi_f\) is a quasi-shuffle morphism, from \Cref{prp:iss_psi} we obtain the following analogue of \Cref{prop:iss.ids}:
\begin{corollary}
    \label{crl:issf.iter}
    Let \(x\) be a \(d\)-dimensional time series and fix \(0\le n\le m\le N\).
    The identities
    \begin{align*}
        \langle\ISS^f(x)_{n,m},u\succ_f v\rangle&= \sum_{n < k \le m}\langle\ISS^f(x)_{n,k-1},u\rangle\,\delta\langle\ISS^f(x)_{n,\cdot},v\rangle_k\\
        \langle\ISS^f(x)_{n,m},u\bullet_f v\rangle&= \sum_{n < k \le m}\delta\langle\ISS^f(x)_{n,\cdot},u\rangle_k\,\delta\langle\ISS^f(x)_{n,\cdot},v\rangle_k
    \end{align*}
    hold for all \(u,v\in H_{\qsh}\).
\end{corollary}

\begin{remark}
    Taking up on \Cref{rmk:algo}, the identities contained in \Cref{crl:issf.iter} also imply an efficient algorithm for the computation of \(\ISS^f\).
\end{remark}

\begin{example}
	Choose for instance \(f(t)=t+t^2\), then
    \begin{align*}
        \langle \ISS^f(x)_{n,m}, [\letter 3][\letter 5] \rangle
        &=
        \langle \vec{\prod_{n< j\le m}}\left( \varepsilon
            + \sum_{a\in \mathfrak A}\delta x_j^a a
        + \left( \sum_{a\in \mathfrak A}\delta x_j^a a \right)^2 \right), [\letter 3][\letter 5]  \rangle \\
        &=
        \sum_{n < i_1 < i_2 \le m}
        \delta x^{[\letter 3]}_{i_1}
        \delta x^{[\letter 5]}_{i_2}
        +
        \sum_{n < i \le m}
        \delta x^{[\letter 3]}_{i}
        \delta x^{[\letter 5]}_{i}
        \\
        &=
        \langle \ISS^f(x)_{n,m}, [\letter 3][\letter 5] + [\letter 3\letter 5] \rangle \\
        &=
        \langle \ISS(x)_{n,m}, \Psi_f( [\letter 3][\letter 5] ) \rangle.
    \end{align*}
    
    Note also the deformed quasi-shuffle computation:
    \begin{align*}
        [\letter 3]\succ_f[\letter 5] &= [\letter 3][\letter 5]-[\letter{35}]\\
        [\letter 5]\succ_f[\letter 3] &= [\letter 5][\letter 3]-[\letter{35}]\\
        [\letter 3]\bullet_f[\letter 5] &= [\letter{35}]
    \end{align*}
    so that
    \[
        [\letter 3]*_f[\letter 5]=[\letter 3][\letter 5]+[\letter 5][\letter 3]-[\letter{35}].
    \]
    Compare this with the undeformed quasi-shuffle product
    \[
        [\letter 3]*[\letter 5]=[\letter 3][\letter 5]+[\letter 5][\letter 3]+[\letter{35}].
    \]
\end{example}

\subsubsection{Relation to F.~Király and H.~Oberhauser}
\label{sss:KO}

We relate our results to the higher-order discrete signatures introduced by F.~Király and H.~Oberhauser \cite[Definition B.4]{KiralyOberhauser2019}.
Given an integer \(p\ge 1\), these authors define for a \(d\)-dimensional time series $x$ the map \(\DS^+_{(p)}(x)\in T(({\R^d}))\) by
\begin{equation}
    \DS^+_{(p)}(x)=\vec{\prod_{0< i\le N}}\sum_{j=0}^p\frac{\left( \mathbf x_i \right)^{\otimes j}}{j!}.
    \label{eq:KO}
\end{equation}
Here $\mathbf{x}_i = \sum_{\letter{j} \in A} \delta x_i^{(j)} [\letter{j}]$ and
$(\mathbf x_i)^{\otimes n}=\sum_{\letter{j}_1,\ldots,\letter{j}_n \in A} \delta x_i^{(j_1)} \cdots \delta x_i^{(j_n)} [\letter{j}_1] \cdots [\letter{j}_n]$.
We note that this map, considered as a linear map on $T(\R^d)$, is not an algebra morphism over
\emph{any} product defined on \(T({\R^d})\) that is compatible with the grading.
Indeed, suppose there is such a product and denote it by \(\circledast\).
Consider the map \(\DS^+_{(p)}(x)\) over a single time step with a non-zero increment.
Fix moreover a single symbol, say \(\letter 1\in A\).
Then
\[
    (\delta x_0^{(1)})^{p+1}=\langle [\letter 1^{\circledast(p+1)}],\DS_{(p)}^+(x)\rangle = 0
\]
which is a contradiction.
This can be resolved by considering the infinite-dimensional polynomial extension of the time series,
including all powers of increments \cite[Example B.2]{TBO2020}.
In essence, this is what the quasi-shuffle approach does
-- the extension is obtained by considering
the bracket terms \([\letter i_1\dotsm\letter i_n]\in\mathfrak A\) and the corresponding extended
increments \(\delta x^{[\letter i_1\dotsm\letter i_n]}\).

However, even when considering the proper extension, the map so obtained does not yield a character
over the quasi-shuffle Hopf algebra when \(p>1\) (the case \(p=1\) corresponds to \(\ISS(x)\), see
\cref{eq:ISS}).
Taking \(p=2\), a single time step and considering the product \([\letter 1]\star[\letter 1]\)
constitutes a simple counterexample.\footnote{The reader is invited to work out the details.}
In this case, the analogue of \cref{eq:KO} consider the full extension equals \(\ISS^{f_p}(x)\), with
\(f_p=t+\frac{1}{2}t^2+\dotsb+\frac{1}{p!}t^p\); \Cref{crl:twchar} restores the character property
of this map, with respect to a \emph{different} product.
Moreover, we have that \(\DS^+_{(p)}(x)=\pi(\ISS^{f_p}(x))\) where \(\pi\colon H'\to H_1'\) is the unique concatenation
morphism satisfying \(\pi([\letter i])=[\letter i]\) and \(\pi([\letter i_1\dotsm\letter i_n])=0\) if \(n>1\).

Finally, we mention that in the limit \(p\to\infty\), \(\DS_{(\infty)}^+(x)\) coincides with the
iterated-integrals signature of the path \(X\) interpolating the values of \(x\) piecewise linearly with
unit speed.
In the same way, the extended version \(\ISS^{f_\infty}(x)\) coincides with the iterated-integrals
signature of an infinite-dimensional extension of \(X\) \cite[Theorem 5.3]{DET2020}.
Both statements are consistent with the fact that \(f_\infty(t)=\exp(t)-1\), so that
\(\Psi_{f_\infty}\) is the Hoffman exponential and \(\star_{f_\infty}\) becomes the shuffle
product (over \(A\) and \(\mathfrak A\), respectively).


\subsection{Polynomial transformation of the time series' increments}
\label{ssec:vecttrans}

\newcommand\ISSPInc[1]{\ISS^{#1(\mathsf{inc})}} 

In some applications, one might have only access to observables of a time series and not to the time
series itself. In others, such as in Machine Learning, applying nonlinearities to the data might be
of use. We introduce now an analogue of the iterated-sums signature, acting on transformed data.
In the following, we will work with different base alphabets, so we explicitly include the size of
it in the notation.
Hence from now on we write e.g. \(\Hqsh(\R^d)\) to indicate this.

Let \(P\colon\R^d\to\R^e\) be a polynomial with \(P(0)=0\), and fix a \(d\)-dimensional time series
\(x\).
We are interested in describing the algebraic properties of the iterated-sums signature of the
transformed increments \((P(\delta x_0),\dotsc,P(\delta x_{N-1}))\).
That is, we wish to study the map
\index[general]{ISSPinc@$\ISSPInc{P}$}
\begin{align}
    \label{eq:Pinc}
        \langle\ISSPInc{P}(x)_{n,m},a_1\dotsm a_p\rangle
        \coloneqq\sum_{n< i_1<\dotsb<i_p\le m}P(\delta x_{i_1})^{a_1}\dotsm P(\delta x_{i_p})^{a_p}.
\end{align}
It is immediate from the definition that, as a word series, \(\ISSPInc{P}_{n,m}(x)\) admits the
factorization
\[
    \ISSPInc{P}(x)_{n,m}=\vec{\prod_{n< j\le m}}\left( \varepsilon+\sum_{a\in\mathfrak A}P(\delta x_j)^a a \right).
\]

In particular we have
\begin{proposition}
        For \(0\le n\le n'\le n''\le N\),
        the identity
        \[
                \ISSPInc{P}(x)_{n,n'}\ISSPInc{P}(x)_{n',n''}=\ISSPInc{P}(x)_{n,n''}
        \]
        holds.
\end{proposition}

\newcommand\diamondH[1]{\Phi^{#1}}
\newcommand\fDiamondH{\diamondH{P}}

Since \(P\) vanishes at 0, the entries of \(\ISSPInc{P}(x)\) are invariant to time-warping.
Therefore, since the iterated-sums signature contains \emph{all} such invariants, we are guaranteed to be
able to express all said entries in terms of those in \(\ISS(x)\).
In order to describe this relation, we consider a map \(\diamondH{P}\colon H(\R^e)\to H(\R^d)\) induced
by \(P\).
Recall that for a multi-index \(\nu=(\nu_1,\dotsc,\nu_d)\in\N_0^d\), and \(x=(x^{(1)},\dotsc,x^{(d)})\in\R^d\) one writes
\(|\nu|\coloneq\nu_1+\dotsb+\nu_d\) and
\[
    x^\nu\coloneq(x^{(1)})^{\nu_1}\dotsm(x^{(d)})^{\nu_d}.
\]
Now, expressing each component of \(P=(p_1,\dotsc,p_e)\) as
\begin{equation}
\label{eq:poly.def}
    p_j(x) = \sum_{|\nu|\le\deg p_j}p_{j;\nu} x^\nu
\end{equation}
we set
\[
    p_\diamond(\letter j)\coloneq\sum_{|\nu|\le\deg p_j}p_{j;\nu}[\letter 1^{\nu_1}\dotsm\letter d^{\nu_d}]
\]
for all \(\letter j\in B=\{\letter 1,\dotsc,\letter e\}\).
Note that this expression can be more succinctly written as
\begin{equation}
\label{eq:nua}
    p_\diamond(\letter j)=\sum_{\substack{a\in\mathfrak A|a|\le\deg p_j}}p_{j;\nu(a)}a
\end{equation}
with \( \nu_j([\letter i_1\dotsm\letter i_n])=\#\{k:\letter i_k=\letter j\} \).

This map extends uniquely to a morphism of commutative algebras, \(p_\diamond\colon S(B)\to S(A)\).
It then has a unique extension $\fDiamondH$
\index[general]{fDiamondH@$\fDiamondH$}
to all of \(H\) as a \emph{concatenation} morphism, i.e.~if
\(w=b_1\dotsm b_m\in\mathfrak B^*\) then
\begin{equation}
    \label{eq:fDiamondH}
    \fDiamondH(w)=p_\diamond(b_1)\dotsm p_\diamond(b_m).
\end{equation}

\begin{example}

    Let $P\colon\R^2 \to \R^3,$ \( P = (p_1, p_2, p_3) = ((x^{(1)})^2, (x^{(2)})^3, x^{(1)} (x^{(2)})^2) \).
    Then
    \begin{align*}
        \diamondH{P}\left( [\letter1] \right)              &= [\letter 1^2] \\
        \diamondH{P}\left( [\letter3^2] \right)            &= [\letter 1^2 \letter 2^4] \\
        \diamondH{P}\left( [\letter 1][\letter3^2] \right) &= [\letter 1^2] [\letter 1^2 \letter 2^4].
    \end{align*}
\end{example}

\begin{lemma}
    The map \(\fDiamondH\colon H_{\qsh}(\R^e)\to H_{\qsh}(\R^d)\) is a quasi-shuffle morphism in the sense of \Cref{def:qshmorph}.
    Moreover, it is a morphism of Hopf algebras.
    \label{lem:fdqsh}
\end{lemma}

\begin{proof}
    By definition, \(\fDiamondH\) preserves the \(\bullet\) product in \(H_{\qsh}\).
    Since it is a concatenation morphism, it must also preserve the half-shuffle \(\succ\).
    Indeed, for \(u\in\mathfrak B^*\) and \(b\in\mathfrak B\) we have that
    \[
        \fDiamondH(u\succ b)=\fDiamondH(ub)=\fDiamondH(u)\fDiamondH(b)=\fDiamondH(u)\succ \fDiamondH(b).
    \]
    By induction, if \(u,v\in\mathfrak B^*\) and \(b\in\mathfrak B\) we have that
    \begin{align*}
        \fDiamondH(u\succ vb)&= \fDiamondH(u\star v)\fDiamondH(b)\\
        &= \fDiamondH(u\succ v+v\succ u+v\bullet u)\fDiamondH(b)\\
        &= \Bigl( \fDiamondH(u)\succ \fDiamondH(v)+\fDiamondH(v)\succ \fDiamondH(u)+\fDiamondH(v)\bullet \fDiamondH(u)
        \Bigr)\fDiamondH(b)\\
        &= (\fDiamondH(u)\star \fDiamondH(v))\fDiamondH(b)\\
        &= \fDiamondH(u)\succ \fDiamondH(v)\fDiamondH(b)\\
        &= \fDiamondH(u)\succ \fDiamondH(vb).
    \end{align*}
    Thus, \(\fDiamondH\) is a quasi-shuffle morphism, and in particular an algebra morphism.

    In order to check that it is also a coalgebra morphism, it suffices to prove the property for \(\letter j\in B\).
    In this case we have that
    \[
        \Delta_{\scriptscriptstyle{\text{dec}}} \fDiamondH([\letter j])=\sum_{|\nu|\le\deg f_j}p_{j;\nu}([\letter 1^{\nu_1}\dotsm\letter
        d^{\nu_d}]\otimes\mathbf e+\mathbf e\otimes[\letter 1^{\nu_1}\dotsm\letter d^{\nu_d}])=\fDiamondH(\letter
        j)\otimes\mathbf e+\mathbf e\otimes \fDiamondH(\letter j).
    \]

\end{proof}

\begin{remark}
    Observe that for \(P\colon\R^e\to\R^f\) and \(Q\colon\R^d\to\R^e\) the composition rule
    \(\diamondH{P\circ Q}=\diamondH{Q}\circ \diamondH{P}\) holds.
    Indeed, for a single letter \(\letter j\in C\coloneq\{\letter 1,\dotsc,\letter p\}\) we see that (using the notation
    introducing in \cref{eq:nua})
    \begin{align*}
        \diamondH{Q}\circ \diamondH{P}([\letter j])
        &= \diamondH{Q}\left( \sum_{a\in\mathfrak A}p_{j;\nu(a)}a \right)\\
        &= \sum_{a=[\letter i_1\dotsm\letter i_n]\in\mathfrak A}p_{j;\nu(a)}[q_\diamond(\letter i_1)\dotsm q_\diamond(\letter i_n)]\\
        &= \sum_{a=[\letter i_1\dotsm\letter i_n]\in\mathfrak A}p_{j;\nu(a)}\prod_{k=1}^n\sum_{b_k\in\mathfrak A}q_{i_k;\nu(b_k)}[b_1\dotsm b_n].
    \end{align*}
    It is straightforward to check, using \cref{eq:poly.def}, that the coefficient of a single letter in the last expression equals the
    coefficient of the corresponding monomial in \(P\circ Q\).
\end{remark}

\begin{theorem}
    Let \(P\colon\R^d\to\R^e\) be a polynomial with vanishing constant coefficient.
    For all \(w\in H(\R^e)\), the relation
    \[
        \langle\ISSPInc{P}(x)_{n,m},w\rangle=\langle\ISS(x)_{n,m},\diamondH{P}(w)\rangle
    \]
    holds.
    In particular, \(\ISSPInc{P}(x)\) is a quasi-shuffle character of \(H_{\qsh}(\R^e)\).
    \label{thm:polymap}
\end{theorem}

\begin{proof}
    We first prove the identity on \(S(B)\). For this, we first show it for \(\letter i\in B\).
    Observe that
    \begin{align*}
        \langle\ISS(x)_{n,m},p_\diamond(\letter i)\rangle&= \sum_{|\nu|\le\deg p_i} p_{i;\nu}\langle \ISS(x)_{n,m},[\letter 1^{\nu_1}\dotsm\letter d^{\nu_d}]\rangle\\
        &= \sum_{|\nu|\le\deg p_i}p_{i;\nu}\sum_{j=n}^{m-1}(\delta x_j^{(1)})^{\nu_1}\dotsm(\delta x_j^{(d)})^{\nu_d}\\
        &= \sum_{n<j\le m}p_i(\delta x_j)\\
        &= \langle \ISSPInc{P}(x)_{n,m},[\letter i]\rangle.
    \end{align*}
    If now \(a=[\letter i_1\dotsm\letter i_p]\in\mathfrak B\) we see that
    \begin{align*}
        \langle\ISS(x)_{n,m},\diamondH{P}(a)\rangle
        &= \langle\ISS(x)_{n,m},[p_\diamond(\letter i_1)\dotsm p_\diamond(\letter i_p)]\rangle\\
        &= \sum_{n<j\le m} p_{i_1}(\delta x_j)\dotsm p_{i_p}(\delta x_j)\\
        &= \langle\ISSPInc{P}(x)_{n,m},a\rangle.
    \end{align*}
    By linearity, the identity holds for all \(a\in S(B)\).

    Finally, by definition, if \(w=a_1\dotsm a_p\in\mathfrak B^*\) then
    \begin{align*}
        \langle\ISS(x)_{n,m},\diamondH{P}(w)\rangle
        &= \langle\ISS(x)_{n,m},\diamondH{P}(a_1\dotsm
        a_{p-1})\diamondH{P}(a_p)\rangle\\
        &= \sum_{n<j\le m} \langle\ISS(x)_{n,j-1},\diamondH{P}(a_1\dotsm a_{p-1})\rangle P(\delta
        x_j)^{a_p}\\
        &= \sum_{n<j\le m} \langle\ISSPInc{P}(x)_{n,j-1},a_1\dotsm a_{p-1}\rangle P(\delta
        x_j)^{a_p}\\
        &= \langle\ISSPInc{P}(x)_{n,m},a_1\dotsm a_p\rangle.
    \end{align*}

    The quasi-shuffle property follows immediately from \Cref{lem:fdqsh}.
\end{proof}

\begin{example}
    In \(d=2\), consider the real-valued polynomial map \(P(x,y)=x^2+y^2=\|(x,y)\|^2\).
    This map can be of interest in a data science context, where one is interested only in the distance of the data
    points to the origin instead of their absolute position in the plane, e.g. if the problem has some rotational
    invariance properties.
We then get the map (with \(e=1\)) \(p_\diamond(\letter 1)=[\letter 1^2]+[\letter 2^2]\). Finally,
    \begin{align*}
        \left\langle \ISSPInc{P}(x)_{n,m}, [\letter 1][\letter 1]\right\rangle&= \sum_{n<i_1<i_2\le m}\left(\left(\delta
    x_{i_1}^{[\letter 1]}\right)^2+\left(\delta x_{i_1}^{[\letter 2]}\right)^2\right)\left(\left(\delta
    x_{i_2}^{[\letter 1]}\right)^2+\left(\delta x_{i_2}^{[\letter 2]}\right)^2\right)\\
    &= \left\langle \ISS(x)_{n,m},([\letter 1^2]+[\letter 2^2])([\letter 1^2]+[\letter 2^2])\right\rangle\\
    &= \left\langle \ISS(x)_{n,m},\diamondH{P}([\letter 1][\letter 1])\right\rangle,
    \end{align*}
    where in the last equality we have used \cref{eq:fDiamondH}.
\end{example}
\begin{remark}
    Even if $d=e$, the map \(\diamondH{P}\) is, in general, not invertible.
    This is due to the fact that \(P^{-1}\) will in general be a formal power series and not just a
    polynomial.
    As an example, consider the single-variable polynomial \(P(x)=x+x^3\).
    Its inverse satisfies \(P^{-1}(x)=x-x^3+3x^5-12x^7+\operatorname{o}(x^7)\) but does not admit a finite series
    representation.
    This means that recovering \(\ISS(x)\) from \(\ISSPInc{P}(x)\) by performing a finite number of operations
    is not always possible. In particular, to recover the increment \(\langle\ISS(x),[\letter 1]\rangle=x_N-x_0\) from
    knowledge of \(\ISSPInc{P}(x)\) one needs to evaluate the series
    \[
        \begin{multlined}
        \langle \ISSPInc{P}(x)_{0,N},[\letter 1]\rangle
        		-3\langle \ISSPInc{P}(x)_{0,N},[\letter 1^3]\rangle\\
		+3\langle \ISSPInc{P}(x)_{0,N},[\letter 1^5]\rangle 
		- 12\langle \ISSPInc{P}(x)_{0,N},[\letter 1^7]\rangle +\dotsb
    \end{multlined}
    \]
    which might not converge depending on the size of the increment \(x_N-x_0\).
    We can say that there is a ``loss of information'', in terms of time-warping invariance, if we are
    only allowed to observe some polynomial transformation of the increments of the data instead of the increments
    themselves.
    
\end{remark}

\subsection{Polynomial transformation of the time series}

We now consider polynomial transformations of the time series itself.
Let \(P\colon\R^d\to\R^e\) be a polynomial map, for some \(e\ge 1\).
We write \( P = (p_1,\dotsc,p_e)\) where \(p_k\in\R[x^{(1)},\dotsc,x^{(d)}]\) is a multivariate polynomial.

Recall from \Cref{sec:quasi-shuffle} that the quasi-shuffle algebra \(H_{\mathrm{qsh}}\) carries a
commutative quasi-shuffle structure.
Moreover, by \Cref{thm:loday} it realizes the free commutative quasi-shuffle over \(\R^d\); in other words, if
\(\tilde H\) is any other commutative quasi-shuffle algebra and \(\Lambda\colon A\to \tilde H\) is a map,
there exists a unique extension \(\Lambda\colon H_{\mathrm{qsh}}\to \tilde H\) respecting the
corresponding quasi-shuffle structures.

Given a time series \(x\), we consider its transform \(X=P(x)\coloneq(P(x_0),\dotsc,P(x_N))\), which
is an \(\R^e\)-valued time series.
Interestingly enough, the iterated-sums signature of \(X\) can be computed just by knowing that of
the untransformed data \(x\).
More precisely we have (cf. \cite[Theorem 2]{CP2020})
\begin{theorem}
	Let $P\colon \R^d \to \R^e,$ be a polynomial map without constant term, i.e., with $P(0) = 0$.
	Given a \(d\)-dimensional time series \(x\) with \(x_0=0\), define the \(e\)-dimensional time series \(X\coloneq
	P(x)\).
	Then, for all \(0\le k\le N\), \(w\in H_{\qsh}\).
  \begin{equation}
    \label{eq:anotherPolynomialMap}
    \langle \ISS(X)_{0,k}, w \rangle = \langle \ISS(x)_{0,k}, \Lambda_P( w ) \rangle,
  \end{equation}
  where $\Lambda_P\colon H_{\mathrm{qsh}}(\R^e)\to H_{\mathrm{qsh}}(\R^d)$ 
  \index[general]{LambdaP@$\Lambda_P$}
	is the unique quasi-shuffle morphism (in the sense of \Cref{def:qshmorph}), determined by its action on $[\letter1], \dots, [\letter{e}]$ as
  \begin{equation*}
    \Lambda_P\left( [\letter{i}] \right) \coloneq \iota( p_i ) \in \Hqsh(\R^d),
  \end{equation*}
	where on the righthand side, \(\iota\colon \R[x^{(1)},\dotsc,x^{(d)}] \to H_{\mathrm{qsh}}(\R^d)\) is the unique morphism of commutative algebras
  satisfying $\iota(x^{(i)}) = [\letter i]$.
	\label{thm:polymap2}
\end{theorem}
\begin{example}
    \label{ex:P}
    Let $P\colon\R^2 \to \R^3,$ \( P = (p_1, p_2, p_3) = ((x^{(1)})^2, (x^{(2)})^3, x^{(1)}
    (x^{(2)})^2) \).
    Then
    \begin{align*}
        \Lambda_P\left( [\letter1] \right) &= [\letter1] \star [\letter1]\\
        &= 2 [\letter{1}] [\letter{1}] + [\letter{1}^2]
        \\
        \Lambda_P\left( [\letter2] \right) &=
        [\letter2] \star [\letter2] \star [\letter2]\\
        &= 6 [\letter{2}][\letter{2}][\letter{2}] + 3 [\letter{2}^2][\letter{2}] + 3 [\letter{2}][\letter{2}^2] + [\letter{2}^3]
        \\
        \Lambda_P\left( [\letter3] \right) &=
        [\letter1] \star [\letter2] \star [\letter2]\\
                                           &=\begin{multlined}[t] 2 [\letter2][\letter2][\letter1]
        + 2 [\letter1][\letter2][\letter2]
        + 2 [\letter2][\letter1][\letter2]
        + 2 [\letter2][\letter1\letter2]
        + 2 [\letter1\letter2][\letter2]\\
        + [\letter1][\letter2^2]
        + [\letter2^2][\letter1]
        + [\letter1\letter2^2].
    \end{multlined}
    \end{align*}
\end{example}

\begin{remark}
    We note that, in general, \(\Lambda_P\) is \emph{not} a morphism of Hopf algebras.
    For example, taking \(P\) as in \Cref{ex:P}, we see that
    \[
        \Delta_{\scriptscriptstyle{\text{dec}}}\Lambda_P([\letter 1])-(\Lambda_P\otimes\Lambda_P)\Delta[\letter1]=2\,[\letter 1]\otimes[\letter 1],
    \]
    that is, \(\Delta_{\scriptscriptstyle{\text{dec}}}\circ\Lambda_P\neq(\Lambda_P\otimes\Lambda_P)\circ\Delta_{\scriptscriptstyle{\text{dec}}}\).
\end{remark}

\begin{proof}
    Since by \Cref{prop:iss.ids} both sides of \cref{eq:anotherPolynomialMap} are quasi-shuffle morphisms, it is enough to show that it holds for letters \([\letter 1], \dotsc,[\letter e]\).
    Now, on one the hand, by definition
    \[
        \langle\ISS(X)_{0,k},[\letter i]\rangle
        = \sum_{0 < j \le k} \delta X_j^{(i)}
        = X^{(i)}_k = p_i(x_k).
    \]
    On the other hand, by definition
    \begin{align*}
        \langle\ISS(x)_{0,k},\Lambda_P([\letter i])\rangle&= \langle\ISS(x)_{0,k},\iota(p_i)\rangle\\
                                                          &= p_i(\langle\ISS(x)_{0,k},[\letter
                                                          1]\rangle,\dotsc,\langle\ISS(x)_{0,k},[\letter d\rangle)\\
        &= p_i(x_k)
    \end{align*}
    and the proof is finished.
\end{proof}
By adding a global shift in the variables, one can also consider polynomials with non-zero constant coefficients (cf. \cite[Corollary 1]{CP2020}).
\begin{corollary}
	Let \(P\colon\R^d\to\R^e\) be a polynomial map, and \(x\) a \(d\)-dimensional time series.
	Define the \(e\)-dimensional time series \(X\coloneq P(x)\), and set \(\tilde P_{x_0}\coloneq
	P(\cdot+x_0)-P(x_0)\).
	Then, for all \(0\le k\le N\), \(w\in H_{\qsh}\).
    \begin{equation*}
      \langle \ISS(X)_{0,k}, w \rangle = \langle \ISS(x)_{0,k}, \Lambda_{\tilde P_{x_0}}( w ) \rangle,
    \end{equation*}
\end{corollary}
\begin{proof}
    The result follows from \Cref{thm:polymap2} and the fact that, if \(\tilde x=x_\cdot-x_0\) then
    \(\tilde x_0=0\) and \(\ISS(\tilde x)_{0,k}=\ISS(x)_{0,k}\) for all \(0\le k\le N\).
\end{proof}


\section{Conclusion / Outlook}
\label{sec:ConcOut}

We have investigated three ways of transforming the iterated-sums signature of a time series.
Using a formal power series $f$,
the iterative definition of the signature
is modified in \Cref{def:generalized},
to obtain a generalized signature.
Its relation to \cite{KiralyOberhauser2019} is sketched in \Cref{sss:KO}.
The transformation can be realized, on the dual side, via a Hopf algebra morphism $\Psi_f$, \Cref{prp:iss_psi}.
Given a polynomial $P$,
directly transforming the increments of the time series
leads to the signature $\ISSPInc{P}$, \eqref{eq:Pinc}.
The transformation can be realized, on the dual side, via a Hopf algebra morphism $\Phi^P$, \Cref{thm:polymap}.
Transforming the time series itself via a polynomial $P$
and calculating its usual signature can also be realized via a
Hopf algebra morphism $\Phi^P$, \Cref{thm:polymap2}.

A remark is in order regarding the three transformations, $\Psi_f$, $\Phi^P$ and $\Lambda_P$, and their properties.
Regarding coalgebra morphisms, the results in \cite{FPT2016} show that they only form a subset of all possible such
morphisms, see \Cref{rem:coalgebraEndomorphisms} and \Cref{rem:specialCase}. On the other hand, we have seen that they
are quasi-shuffle morphisms (in the sense of \Cref{def:qshmorph}), which implies that they satisfy the equations analog
to those in \Cref{prop:iss.ids}.
However, it is worth noting that not all algebra morphisms are necessarily  quasi-shuffle morphisms. Indeed, starting
from the shuffle Hopf algebra $H_\shuffle=(T(V),\shuffle,\Delta_{\scriptscriptstyle{\text{dec}}},\varepsilon,\eta)$, where $T(V)$ is the usual unital tensor
algebra over $V$ and deconcatenation as comultiplication.  We consider $\Phi: H_\shuffle \to H_\shuffle$ defined by
\begin{equation}
\label{comorphism} 
	\Phi = \text{e} + m_{\scriptscriptstyle{\text{con}}}(f \otimes \Phi)\Delta_{\scriptscriptstyle{\text{dec}}},
\end{equation}
where $\text{e}:=\eta \circ \varepsilon$ and the linear map $f$ sends the empty word to zero,  $f(\mathbf{1})=0$, and
any element of $T^+(V):=\bigoplus_{n>0} T_n(V)$ to $T_1(V)$. Here, $m_{\scriptscriptstyle{\text{con}}}$ denotes the
concatenation product on $T(V)$. Following \cite[p.214]{FPT2016}, $\Phi$ is a coalgebra morphism.
For an explicit proof via induction, one can use the fact that $(T(V),m_{\scriptscriptstyle{\text{con}}},\Delta_{\scriptscriptstyle{\text{dec}}},\varepsilon,\eta)$ forms an unital infinitesimal bialgebra \cite{Foissy2009} characterized by the identity
\begin{equation}
\label{infinitesimal} 
	\Delta_{\scriptscriptstyle{\text{dec}}}( w_1 w_2) 
	= (w_1 \otimes \mathrm{id}) \Delta_{\scriptscriptstyle{\text{dec}}}(w_2) 
		+ \Delta_{\scriptscriptstyle{\text{dec}}}(w_1) (\mathrm{id} \otimes w_2) 
		- w_1 \otimes  w_2, 
\end{equation}
for words $w_1,w_2 \in T^+(V)$. If we further assume that $f(\mathbf{1})=0=f(w_1 \shuffle w_2)$ for words $w_1,w_2 \in
T^+(V)$, then  \eqref{comorphism} defines a shuffle algebra morphism, i.e., for $w_1,w_2 \in T^+(V)$, $\Phi(w_1 \shuffle
w_2) = \Phi(w_1) \shuffle \Phi(w_2).$ Note, however, that such a $\Phi$ is not a shuffle, or Zinbiel, morphism
$$
	\Phi(xy)
	= \Phi(x \prec y) 
	= f(xy) + f(x)f(y), 
$$
which is different from $\Phi(x )\prec \Phi(y) = f(x)f(y).$ We can naturally extend the shuffle algebra morphism $\Phi$
to the quasi-shuffle Hopf algebra using Hoffman's exponential, which gives a quasi-shuffle algebra morphism on
$H_{\qsh}$, but not a quasi-shuffle morphism  (in the sense of \Cref{def:qshmorph}).
If we let $\ISS^\Phi = \ISS \circ \Phi$ be the ``transformed'' signature,
this results on $\ISS^\Phi$ \emph{not} satisfying a version \Cref{prop:iss.ids}.
We mention briefly, postponing the
presentation of details to another paper, that transforming $\ISS$ to $\ISS^\Phi$ amounts to translations in the sense
of \cite{Brunedetal2019}.

\subsection{Comparison of the three maps}
\label{ssec:compar}

We lastly consider now transformation maps.
The following is a manifestation of the nonlinear Schur--Weyl duality (see \cite[Proposition 1.3]{FPT2016}).

\begin{proposition}
    For 
    \(f\in t\mathbb R\dbra{t}\) a formal diffeomorphism
    and \(P\colon\R^d\to\R^e\) a polynomial with \(P(0)=0\),
    \begin{align*}
        \Psi_f \circ \diamondH{P} = \diamondH{P} \circ \Psi_f.
    \end{align*}
\end{proposition}

\begin{proof}
    This follows from the fact that for any word $w$ and any $I \in C(\ell(w))$ (defined in \Cref{sss:coalg.pseries})
    we have
    \begin{align*}
        \diamondH{P}( I[w] )
        =
        I[ \diamondH{P}( w )].
    \end{align*}
\end{proof}

The polynomial transformation of the time series' increments and the polynomial transformation of the time series are related as follows.

\begin{proposition}
    We observe that for any polynomial \(P\) with \(P(0)=0\), the map \(P_\diamond: S(B) \to S(A)\) of
    \Cref{ssec:vecttrans} can be recovered from \(\Lambda_P\) by post-composing with the projection map \(\pi:H\to S(A)\).
    That is
    \begin{align*}
        P_\diamond = \pi \circ \Lambda_P.
    \end{align*}
\end{proposition}

\bibliographystyle{../tropical-iss/arxivalpha}
\bibliography{cutoff}

\printindex[general]

\end{document}